\documentclass [a4paper,12pt]{amsart}
\usepackage{graphicx}
\usepackage[ansinew]{inputenc}    
\usepackage[all]{xy}

\usepackage{amssymb,amscd}

\usepackage{enumitem}

\newtheorem{thm}{Theorem}[section]

\def\im{\operatorname{Im}}
\def\Ass{\operatorname{Ass}}

\def\Supp{\operatorname{Supp}}
\def\coker{\operatorname{coker}}

\def\C{\mathbb C}

\def\dim{\operatorname{dim}}

\def\coker{\operatorname{coker}}

\usepackage{graphicx}
\usepackage[ansinew]{inputenc}    
\usepackage[all]{xy}
\usepackage{amssymb,amscd}

\newtheorem{cor}[thm]{Corollary}
\newtheorem{teo}[thm]{Theorem}
\newtheorem{lem}[thm]{Lemma}
\newtheorem{prop}[thm]{Proposition}

\theoremstyle{definition}
\newtheorem{ex}[thm]{Example}
\newtheorem{defi}[thm]{Definition}

\newtheorem{remark}[thm]{Remark}

\def\C{\mathbb C}

\def\dim{\operatorname{dim}}

\def\coker{\operatorname{coker}}

\usepackage{color}

 \everymath{\displaystyle}
\begin{document}

\title[The relative Bruce-Roberts number of a function on a hypersurface]{The relative Bruce-Roberts number of a function on a hypersurface}

\author{ B. K. Lima-Pereira, J.J. Nu\~no-Ballesteros, B. Or\'efice-Okamoto, J.N. Tomazella}

\address{Departamento de Matem\'atica, Universidade Federal de S\~ao Carlos, Caixa Postal 676,
	13560-905, S\~ao Carlos, SP, BRAZIL}
\email{barbarapereira@dm.ufscar.br}

\address{Departament de Matem\`atiques,
Universitat de Val\`encia, Campus de Burjassot, 46100 Burjassot
SPAIN. Departamento de Matemática, Universidade Federal da Paraíba
		CEP 58051-900, João Pessoa - PB, Brazil}
\email{Juan.Nuno@uv.es}

\address{Departamento de Matem\'atica, Universidade Federal de S\~ao Carlos, Caixa Postal 676,
	13560-905, S\~ao Carlos, SP, BRAZIL}
\email{bruna@dm.ufscar.br}

\address{Departamento de Matem\'atica, Universidade Federal de S\~ao Carlos, Caixa Postal 676,
	13560-905, S\~ao Carlos, SP, BRAZIL}

\email{tomazella@dm.ufscar.br}

	\thanks{The first author was partially supported by CAPES.
		The second author was partially supported by MICINN Grant PGC2018--094889--B--I00 and by GVA Grant AICO/2019/024. The third author was partially supported by FAPESP Grant
		2016/25730-0. The fourth author was partially supported by CNPq Grant
		309086/2017-5 and FAPESP Grant 2018/22090-5.}
		

\maketitle

\begin{abstract}
We consider the relative Bruce-Roberts number $\mu_{BR}^{-}(f,X)$ of a function on an isolated hypersurface singularity $(X,0)$. We show that $\mu_{BR}^{-}(f,X)$ is equal to the sum of the Milnor number of the fibre $\mu(f^{-1}(0)\cap X,0)$ plus the difference $\mu(X,0)-\tau(X,0)$ between the Milnor and the Tjurina numbers of $(X,0)$. As an application, we show that the usual Bruce-Roberts number $\mu_{BR}(f,X)$ is equal to $\mu(f)+\mu_{BR}^{-}(f,X)$. We also deduce that the relative logarithmic characteristic variety $LC(X)^-$, obtained from the logarithmic characteristic variety $LC(X)$ by eliminating the component corresponding to the complement of $X$ in the ambient space, is Cohen-Macaulay.
\end{abstract}


\section{Introduction}

Let $(X,0)$ be a germ of complex analytic set in $\C^n$ and $f:(\C^n,0)\to(\C,0)$ a holomorphic function germ. The Bruce-Roberts number of $f$ with respect to $(X,0)$ was introduced by Bruce and Roberts in \cite{bruce roberts} and is defined as
\[
\mu_{BR}(f,X)=\dim_{\C}\frac{\mathcal{O}_{n}}{df(\Theta_{X})},
\]
where $\mathcal O_n$ is the local ring of holomorphic functions $(\C^n,0)\to\C$, $df$ is the differential of $f$ and $\Theta_X$ is the $\mathcal O_n$-submodule of $\Theta_n$ of vector fields on $(\C^n,0)$ which are tangent to $(X,0)$ at its regular points. If $I_X$ is the ideal of $\mathcal O_n$ of functions vanishing on $(X,0)$, then
\[
\Theta_{X}=\{\xi\in\Theta_{n}\ |\ dh (\xi)\in I_X,\ \forall h\in I_X \}.
\]
In particular, when $X=\C^n$, $df(\Theta_{X})$ is the Jacobian ideal of $f$ and thus, $\mu_{BR}(f,X)$ coincides with the classical Milnor number $\mu(f)$. 
We remark that $\Theta_X$ is also denoted in some papers by $\mbox{Der}(-\log X)$, following Saito's notation \cite{saito}.
The main properties of $\mu_{BR}(f,X)$ are the following (see \cite{bruce roberts}):
\begin{enumerate}
\item[(a)] $\mu_{BR}(f,X)$ is invariant under the action of the group $\mathcal R_X$ of diffeomorphisms $\phi:(\C^n,0)\to(\C^n,0)$ which preserve $(X,0)$;
\item[(b)] $\mu_{BR}(f,X)<\infty$ if and only if $f$ is finitely determined with respect to the $\mathcal R_X$-equivalence;
\item[(c)] $\mu_{BR}(f,X)<\infty$ if and only if $f$ restricted to each logarithmic stratum is a submersion in a punctured neighbourhood of the origin.
\end{enumerate}

In general, $\mu_{BR}(f,X)$ is not so easy to compute as the classical Milnor number. The main difficulty comes from the computation of the module $\Theta_X$ and most of the times it is necessary the use of a symbolic computer system like {\sc Singular} \cite{singular}. When $\mu_{BR}(f,X)$ is finite, the fibre $(f^{-1}(0)\cap X,0)$ is an isolated complete intersection singularity (ICIS), therefore it has well defined Milnor number. In a previous paper \cite{nunoballesteros oreficeokamoto limapereira tomazela} we considered the case that $(X,0)$ is an isolated hypersurface singularity (IHS). We showed that 
\begin{equation}\label{BR-formula}
\mu_{BR}(f,X)=\mu(f)+\mu(f^{-1}(0)\cap X,0)+\mu(X,0)-\tau(X,0),
\end{equation}
where $\mu$ and $\tau$ are the  Milnor and the Tjurina numbers, respectively. Thus, \eqref{BR-formula} gives an easy way to compute $\mu_{BR}(f,X)$ in terms of well known invariants. The formula \eqref{BR-formula} was also obtained independently in \cite{kourliouros} and previously in \cite{nunoballesteros oreficeokamoto tomazella} when $(X,0)$ is weighted homogeneous.

An important application of \eqref{BR-formula} allowed us to conclude in \cite{nunoballesteros oreficeokamoto limapereira tomazela} that the logarithmic characteristic variety $LC(X)$ is Cohen-Macaulay. We recall that $LC(X)$ is the subvariety of the cotangent bundle $T^*\C^n$ of pairs $(x,\alpha)$ such that $\alpha(\xi_x)=0$, for all $\xi\in \Theta_X$ and for all $x$ in a neighbourhood of $0$. Cohen-Macaulayness of $LC(X)$ is equivalent to that $\mu_{BR}(f,X)$ satisfies the principle of conservation for deformations of $f$ (with a fixed $X$). In particular, if $f_t$ is a Morsification of $f$, we have
\[
\mu_{BR}(f,X)=\sum_{\alpha} m_\alpha n_\alpha,
\]
where $n_\alpha$ is the number of critical points of $f_t$ restricted to each logarithmic stratum $X_\alpha$ and $m_\alpha$ is the multiplicity of $LC(X)$ along the irreducible component $Y_\alpha$ associated with $X_\alpha$. When $(X,0)$ is an IHS it always has a finite number of logarithmic strata (i.e., it is holonomic in Saito's terminology) given by $X_0=\C^n\setminus X$, $X_i\setminus\{0\}$, with $i=1,\dots,k$ and $X_{k+1}=\{0\}$, where $X_1,\dots,X_k$ are the irreducible components of $X$ at $0$. 

In this paper we are interested in another important invariant introduced in \cite{bruce roberts},
\[
\mu_{BR}^{-}(f,X)=\dim_{\C}\frac{\mathcal{O}_{n}}{df(\Theta_{X})+I_X},
\]
which we call here the relative Bruce-Roberts number. This is an invariant of the restricted function $f:(X,0)\to(\C,0)$ under the induced $\mathcal R_X$-action. In fact, as commented in \cite{bruce roberts}, it is equal to the codimension of the $\mathcal R_X$-orbit. Moreover, $\mu_{BR}^{-}(f,X)$ is finite if and only if $f$ restricted to each logarithmic stratum (excluding $X_0$) is a submersion in a punctured neighbourhood of the origin.

A natural question is about the relationship between $\mu_{BR}(f,X)$ and $\mu_{BR}^{-}(f,X)$. It is showed in \cite{bruce roberts} that if $(X,0)$ is a weighted homogeneous ICIS then
\[
\mu_{BR}^{-}(f,X)=\mu(f^{-1}(0)\cap X,0).
\]
This, combined with \eqref{BR-formula} when $(X,0)$ is a weighted homogeneous IHS, gives that
\begin{equation}\label{BR-formula2}
\mu_{BR}(f,X)=\mu(f)+\mu_{BR}^{-}(f,X).
\end{equation}

Our main result in Section 2 is that if $(X,0)$ is any IHS and $\mu_{BR}^{-}(f,X)$ is finite, then 
\begin{equation}\label{BR-formula3}
\mu_{BR}^{-}(f,X)=\mu(f^{-1}(0)\cap X,0)+\mu(X,0)-\tau(X,0).
\end{equation}
In particular, \eqref{BR-formula2} also holds when $\mu_{BR}(f,X)$ is finite, even if $(X,0)$ is not weighted homogeneous. We also show in Example \ref{ex:muf} that \eqref{BR-formula2} is not true for higher codimension ICIS.

The relative logarithmic characteristic variety $LC(X)^{-}$ is obtained from $LC(X)$ by eliminating the component $Y_0$ associated with the stratum $X_0=\C^n\setminus X$. In \cite{bruce roberts}, they showed that $LC(X)$ is never Cohen-Macaulay when $(X,0)$ has codimension $>1$ along the points on $X_0$, but $LC(X)^{-}$ is always Cohen-Macaulay when $(X,0)$ is a weighted homogeneous ICIS (of any codimension). Again, Cohen-Macaulayness of $LC(X)^{-}$ is interesting since it implies that
\[
\mu_{BR}^{-}(f,X)=\sum_{\alpha\ne0} m_\alpha n_\alpha,
\]
for any Morsification $f_t$ of $f$. As an application of \eqref{BR-formula3} we show in Section 3 that $LC(X)^{-}$ is also Cohen-Macaulay for any IHS $(X,0)$ (not necessarily weighted homogeneous).

In Section 4, we consider any holonomic variety $(X,0)$ and study characterisations of Cohen-Macaulayness of $LC(X)$ and $LC(X)^{-}$ in terms of the relative polar curve associated with a Morsification $f_t$ of $f$. Finally, in Section 5 we give a formula which generalises the classical Thom-Sebastiani formula for the Milnor number of a function defined as a sum of functions with separated variables.

\section{The Relative Bruce-Roberts Number}
The main goal of this section is to prove the equality (\ref{BR-formula3}). The next lemma is inspired in \cite[Proposition 2.8]{bivia ruas} and the proof follows by using the same ideas.
\begin{lem}\label{trivial vectors fields}
Let $(X,0)$ be an IHS determined by $\phi:(\C^{n},0)\to(\C,0)$ and $f\in\mathcal{O}_{n}$.
The map $(\phi,f):(\C^{n},0)\to(\C^{2},0)$ defines an ICIS if and only if $\mu_{BR}^{-}(f,X)<\infty$.
\end{lem}
%

%

The following technical lemma will be used in the proof of the next theorem. Given a matrix $A$ with entries in a ring $R$, we denote by $I_k(A)$ the ideal in $R$ generated the $k\times k$ minors of $A$.

\begin{lem}\label{lema do juanjo}
Let $f,g\in\mathcal{O}_n$ be such that $\dim V(J(f,g))=1$ and $V(Jf)=\{0\}$, and consider the following matrices
$$A=\begin{pmatrix}
\tfrac{\partial f}{\partial x_{1}}&...&\tfrac{\partial f}{\partial x_{n}}\\
\tfrac{\partial g}{\partial x_{1}}&...&\tfrac{\partial g}{\partial x_{n}}
\end{pmatrix},
\quad
A'=\begin{pmatrix}
\mu&\tfrac{\partial f}{\partial x_{1}}&...&\tfrac{\partial f}{\partial x_{n}}\\
\lambda&\tfrac{\partial g}{\partial x_{1}}&...&\tfrac{\partial g}{\partial x_{n}}
\end{pmatrix},$$
where $\lambda,\mu\in \mathcal{O}_n$.
Let $M,M'$ be the submodules of $\mathcal O_n^2$ generated by the columns of  $A,A'$ respectively.
If $I_2(A)=I_2(A')$ then $M=M'$.

\end{lem}
\begin{proof}
We see $A$ and $A'$ as homomorphims of modules over $R:=\mathcal O_n$:
\[
A\colon R^n\longrightarrow R^2,\quad A'\colon R^{n+1}\longrightarrow R^2.
\]
We consider the $R$-module $R^{2}/M=\coker(A)$, which has support $V(I_2(A))=V(J(f,g))$. Therefore, $\dim(R^2/M)=1=n-(n-2+1)$ and hence it is Cohen-Macaulay (see \cite{BR}). In particular, it is unmixed.
Now, $M'/M$ is a submodule of $R^2/M$, so the associated primes $\Ass (M'/M)$ are included in $\Ass(R^{2}/M)$. If $M'/M\ne0$ then $\Ass (M'/M)\ne\emptyset$ and it follows that $\dim(M'/M)=1$.

Let $U$ be a neighbourhood of $0$ in $\C^{n}$ such that $0$ is the only critical point of $f$. For all  $x\in U\setminus\{0\}$, there exist $i_{0}\in\{1,...,n\}$, such that $\partial f/\partial x_{i_{0}}(x)\neq 0$.
We may suppose $i_{0}=1$.
Making elementary column operations in the matrices $A$ and $A'$
we obtain  
$$B=\begin{pmatrix}
1  &0&\dots   &  0 \\
c_1&c_{2}& \dots & c_n  
\end{pmatrix},
\quad
B'=\begin{pmatrix}
\mu & 1 &0&\dots   &  0 \\
\lambda & c_1&c_{2}  &  \dots & c_n  
\end{pmatrix}$$
such that $$I_{2}(A)=I_{2}(B),\;I_{2}(A')=I_{2}(B'),\;\im(A)=\im(B)\textup{ and }\im(A')=\im(B').$$

By hypothesis $I_{2}(A)=I_{2}(A')$ and consequently  $\langle c_{2},...,c_{n}\rangle=\langle\mu c_{1}-\lambda,c_{2},...,c_{n}\rangle.$
This implies $\lambda=\mu c_{1}+\alpha_{2}c_{2}+...+\alpha_{n}c_{n}$, for some
$\alpha_{2},...,\alpha_{n}\in R$. Thus, 
\[
\left(\begin{array}{c} \mu \\ \lambda\end{array}\right)=\mu \left(\begin{array}{c} 1 \\ c_1\end{array}\right)+
\alpha_2 \left(\begin{array}{c} 0 \\ c_2\end{array}\right)+\dots+
\alpha_n \left(\begin{array}{c} 0 \\ c_n\end{array}\right).
\]
 and hence $\left(M'/M\right)_{x}=0$. This shows that $\Supp(M'/M)\subset\{0\}$ and hence, $M'=M$.
\end{proof}


Given an IHS $(X,0)$ defined by a holomorphic function germ $\phi:(\C^{n},0)\to(\C,0),$ we consider the  $\mathcal{O}_{n}$-submodule of the trivial vectors fields, denoted by  $\Theta_{X}^{T}$, generated by  $$\phi\frac{\partial}{\partial x_{i}},\frac{\partial\phi}{\partial x_{j}}\frac{\partial}{\partial x_{k}}-\frac{\partial\phi}{\partial x_{k}}\frac{\partial}{\partial x_{j}},\text{ with }i,j,k=1,...,n;k\neq j.$$ This module was related to the Tjurina number of $(X,0)$ in  \cite{nunoballesteros oreficeokamoto limapereira tomazela, tajima}. By using different approaches, it is showed that $\tau(X,0)=\dim_{\C}\Theta_{X}/\Theta_{X}^{T}$. Moreover, in \cite{nunoballesteros oreficeokamoto limapereira tomazela} we also proved that $\tau(X,0)=\dim_{\C}df(\Theta_{X})/df(\Theta_{X}^{T})$ where $f$ is any $\mathcal{R}_{X}$-finitely determined function germ. The following result generalizes this equality with weaker hypothesis on $f$.

%

\begin{teo}\label{resultados juntos}
Let $(X,0)$ be an IHS determined by $\phi:(\C^{n},0)\to(\C,0)$ and $f\in\mathcal{O}_{n}$ such that $\mu_{BR}^{-}(f,X)<\infty$, then:
\begin{enumerate}[label=\emph{(\roman*)}]
\medskip
\item $\frac{\Theta_{X}}{\Theta_{X}^{T}}\approx\frac{df(\Theta_{X})+I_{X}}{df(\Theta_{X}^{T})+I_{X}};$

\medskip
\item $\frac{\Theta_{X}}{\Theta_{X}^{T}}\approx\frac{df(\Theta_{X})}{df(\Theta_{X}^{T})};$

\medskip
\item $df(\Theta_{X})\cap I_{X}=JfI_{X};$

\medskip
\item $\frac{\mathcal{O}_{n}}{Jf}\approx\frac{df(\Theta_{X}^{-})}{df(\Theta_{X})};$

\medskip
\item $df(\Theta_{X}):I_{X}=Jf;$

\medskip
\item $df(\Theta_{X}^{T}):I_{X}=Jf,$
\end{enumerate}
where $I_{X}$ is the ideal generated by $\phi$.
\end{teo}
\begin{proof}
(i) The homomorphism $\Psi:\Theta_{X}\to df(\Theta_{X})+I_{X}$ defined by $\Psi(\xi)=df(\xi)$
induces the isomorphism
%
 $$\overline{\Psi}:\frac{\Theta_{X}}{\Theta_{X}^{T}}\to\frac{df(\Theta_{X})+I_{X}}{df(\Theta_{X}^{T})+I_{X}}.$$
 In fact, it is enough to show that $\Psi^{-1}(df(\Theta_{X}^{T})+I_{X})\subset \Theta_{X}^{T}.$ Let $\xi\in\Psi^{-1}(df(\Theta_{X}^{T})+I_{X})$ then $\Psi(\xi)\in df(\Theta_{X}^{T})+I_{X}$, that is, there exist $\eta \in\Theta_{X}^{T}$ and $\mu, \lambda\in\mathcal{O}_{n}$, such that

$$\left\{ \begin{array}{cc}
df(\xi-\eta)=\mu\phi\\
d\phi(\xi-\eta)=\lambda\phi
\end{array}\right.,$$
then  
$$\left(\begin{array}{c}
\mu\phi\\
\lambda\phi
\end{array}\right)\in\left\langle\left(\begin{array}{c}
\tfrac{\partial f}{\partial x_{i}}\vspace{0.1cm}\\
\tfrac{\partial \phi}{\partial x_{i}}
\end{array}\right) \; i=1,...,n\right\rangle$$
 and $$I_{2}\begin{pmatrix}\mu\phi&\tfrac{\partial f}{\partial x_{1}}&...&\tfrac{\partial f}{\partial x_{n}}\vspace{0.1cm}\\
 \lambda\phi&\tfrac{\partial \phi}{\partial x_{1}}&...&\tfrac{\partial \phi}{\partial x_{n}} \end{pmatrix}=I_{2}\begin{pmatrix}\tfrac{\partial f}{\partial x_{1}}&...&\tfrac{\partial f}{\partial x_{n}}\vspace{0.1cm}\\
 \tfrac{\partial \phi}{\partial x_{1}}&...&\tfrac{\partial \phi}{\partial x_{n}} \end{pmatrix}=J(f,\phi).$$
 
 Therefore $$\left|\begin{array}{cc}
\mu&\tfrac{\partial f}{\partial x_{i}}\vspace{0.1cm}\\
\lambda&\tfrac{\partial \phi}{\partial x_{i}}
\end{array}\right|\phi\in J(f,\phi)$$
and since $\phi$ is regular in $\frac{\mathcal{O}_{n}}{J(f,\phi)}$ then
 $$\left|\begin{array}{cc}
\mu&\tfrac{\partial f}{\partial x_{i}}\vspace{0.1cm}\\
\lambda&\tfrac{\partial \phi}{\partial x_{i}}
\end{array}\right|\in J(f,\phi),\; i=1,...,n.$$

By Lemma \ref{lema do juanjo}, $\lambda\in J\phi$ and using \cite[Lemma 3.1]{nunoballesteros oreficeokamoto limapereira tomazela}, $\xi\in\Theta_{X}^{T}.$


\medskip
(ii) This equality also was proved in \cite{nunoballesteros oreficeokamoto limapereira tomazela} with the additional hypothesis that $f$ is $\mathcal{R}_{X}$-finitely determined.  

The epimorphism $\psi:\Theta_{X}\to df(\Theta_{X})$ defined by $\psi(\xi)=df(\xi)$
induces the ismorphism

 \begin{align*}
\overline{\psi}:\frac{\Theta_{X}}{\Theta_{X}^{T}}&\to \frac{df(\Theta_{X})}{df(\Theta_{X}^{T})}.\\
\end{align*}
In fact, let $\xi\in\ker(\psi)$, then there exist $\lambda\in\mathcal{O}_{n}$, such that 
$$\left\{ \begin{array}{cc}
df(\xi)=0\\
d\phi(\xi)=\lambda\phi
\end{array}\right.$$
The rest is similar to the proof of (i).

\medskip
(iii)
Let $\xi\in\Theta_{X}$ be such that $df(\xi)\in I_{X}$, then there exist $\mu,\lambda\in\mathcal{O}_{n}$, such that 

$$\left\{ \begin{array}{cc}
df(\xi)=\mu\phi\\
d\phi(\xi)=\lambda\phi
\end{array}\right.$$
Using the same techniques of the proof of (i) we have
%
%
%
%
%
$$df(\Theta_{X})\cap I_{X}\subset Jf I_{X}.$$ The other inclusion is immediate.

\medskip
(iv) It follows from the isomomorphisms  $$\frac{df(\Theta_{X}^{-})}{df(\Theta_{X})}=\frac{df(\Theta_{X})+I_{X}}{df(\Theta_{X})}\approx \frac{I_{X}}{df(\Theta_{X})\cap I_{X}}\stackrel{(iii)}{=}\frac{I_{X}}{JfI_{X} }\approx\frac{\mathcal{O}_{n}}{Jf}.$$

\medskip
(v) It follows from (iii).

\medskip
(vi) It follows from (v) and $Jf\subset df(\Theta_{X}^{T}):I_{X}$.

\end{proof}

\begin{remark}The items (ii) and (iv) of Theorem \ref{resultados juntos} seem a bit peculiar since  from (iv) the quotient $df(\Theta_{X}^{-})/df(\Theta_{X})$ does not depend on $(X,0)$ while from (ii), $df(\Theta_{X})/df(\Theta_{X}^{T})$ does not depend on $f$. Moreover by \cite{nunoballesteros oreficeokamoto limapereira tomazela, tajima} if $(X,0)$ is an IHS determined by $\phi:(\C^{n},0)\to(\C,0)$, then $\dim_{\C}\tfrac{\Theta_{X}}{\Theta_{X}^{T}}=\tau(X,0),$ therefore   $$\dim_{\C}\frac{df(\Theta_{X})+I_{X}}{df(\Theta_{X}^{T})+I_{X}}=\dim_{\C}\frac{df(\Theta_{X})}{df(\Theta_{X}^{T})}=\tau(X,0).$$  \end{remark}

%
%


The next theorem is one of the main results of this work.

\begin{teo}\label{caracterizacao para o numero de bruce roberts extendido}
Let $(X,0)$ is an IHS determined by $\phi:(\C^{n},0)\to(\C,0)$ and $f\in\mathcal{O}_{n}$ a function germ such that $\mu_{BR}^{-}(f,X)<\infty$. Then $(\phi,f)$ defines an ICIS and $$\mu(f^{-1}(0)\cap X,0)=\mu_{BR}^{-}(f,X)+\tau(X,0)-\mu(X,0).$$
\end{teo}
\begin{proof}
We consider the exact sequence
$$0\longrightarrow \frac{df(\Theta_{X}^{-})}{df(\Theta_{X}^{T})+I_{X}}\stackrel{i}{\longrightarrow}\frac{\mathcal{O}_{n}}{df(\Theta_{X}^{T})+I_{X}}\stackrel{\pi}{\longrightarrow}\frac{\mathcal{O}_{n}}{df(\Theta_{X}^{-})}\longrightarrow 0.$$
Since $(X,0)$ is an IHS $$df(\Theta_{X}^{T})=J(f,\phi)+JfI_{X},$$ hence
\begin{align*}
\mu_{BR}^{-}(f,X)&=\dim_{\C}\frac{\mathcal{O}_{n}}{J(f,\phi)+I_{X}}-\dim_{\C}\frac{df(\Theta_{X})+I_{X}}{df(\Theta_{X}^{T})+I_{X}}\\
                &=\mu(f^{-1}(0)\cap X,0)+\mu(X,0)-\tau(X,0).
\end{align*}
The last equality is a consequence of the Lê-Greuel formula \cite{brieskorn greuel} and Theorem \ref{resultados juntos} (i).
\end{proof}



\section{The Relative Bruce-Roberts number of a function with isolated singularity}

In this section, $(X,0)$ is an IHS and $f\in\mathcal{O}_{n}$ is a function germ $\mathcal{R}_{X}$-finitely determined, then all the results in the previous section are true in this case. In particular from  (iv) of Theorem \ref{resultados juntos}

\begin{equation}\label{essa} \mu(f)=\dim_{\C}\frac{df(\Theta_{X}^{-})}{df(\Theta_{X})}.
\end{equation}
Therefore, by the exact sequence
$$0\longrightarrow \frac{df(\Theta_{X}^{-})}{df(\Theta_{X})}\stackrel{i}{\longrightarrow}\frac{\mathcal{O}_{n}}{df(\Theta_{X}}\stackrel{\pi}{\longrightarrow}\frac{\mathcal{O}_{n}}{df(\Theta_{X}^{-})}\longrightarrow 0,$$
we conclude that

$$\mu_{BR}(f,X)=\mu(f)+\mu_{BR}^{-}(f,X).$$

 The following example shows that the characterization of the Milnor number (\ref{essa}) is not true anymore when $(X,0)$ is an ICIS with codimension higher than one.  
\begin{ex}\label{ex:muf}
Let $(X,0)$ be an ICIS determined by $\phi(x,y,z)=(x^3+x^2y^2+y^7+z^3,xyz)$, and $f(x,y,z)=xy-z^4$,  $f$ is a $\mathcal{R}_{X}$-finitely determined and $$3=\mu(f)\neq\dim_{\C}\frac{df(\Theta_{X}^{-})}{df(\Theta_{X})}=6.$$
\end{ex}

As a consequence of the characterization of the Milnor number (\ref{essa}), we prove that $LC(X)^{-}$ is Cohen-Macaulay when $(X,0)$ is an IHS.

The logarithmic characteristic variety, $LC(X)$, is defined as follows. 
Suppose the vector fields $\delta_1,\dots ,\delta_m$ generate $\Theta_X$ on some neighborhood $U$ of $0$ in $\C^n$. Let $T^*_U\C^n$ be the restriction of the cotangent bundle of $\C^n$ to $U$. We define $LC_U(X)$ to be 
$$
LC_U(X)=\{(x,\xi)\in T^*_U\C^n:\xi(\delta_i(x))=0, i=1,\dots ,m\}.
$$ 
Then $LC(X)$ is the germ of $LC_U(X)$ in $T^*\C^n$ along $T^*_0\C^n$, the cotangent space to $\C^n$ at $0$.
As $LC(X)$ is independent of the choice of the vector fields $\delta_i$ then it is a well defined germ of analytic subvariety in $T^*\C^n$ (see \cite{bruce roberts,saito}).
  
If $(X,0)$ is holonomic with logarithmic strata $X_0,\dots,X_k$ then $LC(X)$ has dimension $n$, and its irreducible components are $Y_0,\dots,Y_k$, with $Y_i=\overline{N^*X_i}$ as set-germs, where $\overline{N^*X_i}$ is the closure of the conormal bundle $N^*X_i$ of $X_i$ in $\C^n$ (see \cite[Proposition 1.14]{bruce roberts}). 

%

When $(X,0)$ has codimension higher than one Bruce and Roberts proved that $LC(X)$ is not Cohen-Macaulay. Then they consider the subspace of $LC(X)$ obtained by deleting the component $Y_{0}$, that correspond to the stratum $X_0= \C^{n}\setminus X$, that is  $$LC(X)^{-}=\bigcup_{i=1}^{k+1}Y_{i}$$ and as set-germs,
$$LC(X)^{-}=\bigcup_{i=1}^{k+1}\overline{N^{*}X_{i}}.$$
An interesting fact about $LC(X)^{-}$ is that it may be Cohen-Macaulay even when $LC(X)$ is not Cohen-Macaulay, for example
if $(X,0)$ is a weighted homogeneous ICIS, then $LC(X)^{-}$ is Cohen-Macaulay, \cite{bruce  roberts}. 

%
%

\begin{prop}
Let $(X,0)$ be an IHS, then $LC(X)^{-}$ is  Cohen-Macaulay.
\end{prop} 
\begin{proof}
We consider $(0,p)\in LC(X)^{-}$, then $(0,p)\in LC(X)$ and there exists $f\in\mathcal{O}_{n}$ such that $df(0)=p$. In \cite{nunoballesteros oreficeokamoto limapereira tomazela} we proved that $LC(X)$ is Cohen-Macaulay. Therefore, by \cite[Proposition 5.8]{bruce roberts}, $$\mu_{BR}(f,X)=\sum_{i=0}^{k+1}m_{i}n_{i}=m_{0}n_{0}+\sum_{i=1}^{k+1}m_{i}n_{i}=\mu(f)+\sum_{i=1}^{k+1}m_{i}n_{i}.$$  
where $n_{i}$ is the number of critical points of a Morsification of $f$ in $X_{i}$ and $m_{i}$ is the multiplicity of irreducible component $Y_{i}$. Thus, $$\mu_{BR}^{-}(f,X)=\mu_{BR}(f,X)-\dim_{\C}\frac{df(\Theta_{X}^{-})}{df(\Theta_{X})}=\mu_{BR}(f,X)-\mu(f)=\sum_{i=1}^{k+1}m_{i}n_{i}.$$
 and by \cite[Proposition 5.11]{bruce roberts} we obtain that $LC(X)^{-}$ is Cohen-Macaulay.
\end{proof}

\begin{remark} We remark that in the proof of the previous proposition we just used that if $(X,0)\subset(\C^{n},0)$ is a hypersurface such that $\dim_{\C}df(\Theta_{X}^{-})/df(\Theta_{X})=\mu(f)$ for all $f$ $\mathcal{R}_{X}$ finitely determined then  $LC(X)^{-}$ is Cohen-Macaulay if and only if $LC(X)$ is Cohen-Macaulay.   \end{remark}

\section{Polar Curves and Logarithmic Characteristic Varieties}

It is important to know if the logarithmic characteristic variety of an analytic variety is Cohen-Macaulay. In \cite{nunoballesteros oreficeokamoto limapereira tomazela} we showed that this is the case for IHS. For non-isolated singularities it is an open problem. 
In this section we give one more step in order to solve it: we study the polar curve and the relative polar curve of a holomorphic function germ over a holonomic analytic variety. We show that these curves are Cohen-Macaulay if and only if the logarithmic characteristic variety and the relative logarithmic characteristic variety (respectively) are Cohen-Macaulay.
As a consequence, we have the principle of conservation for the Bruce-Roberts number. 

\begin{defi}\label{curva polar}
Let $f\in\mathcal{O}_{n}$ be a $\mathcal{R}_{X}$-finitely determined function germ and %
$F:(\C^{n}\times\C,0)\to(\C,0)$, $F(t,x)=f_t(x)$,
a 1-parameter deformation of $f$. The \emph{polar curve} of $F$ in $(X,0)$ is
$$C=\{(x,t)\in\C^{n}\times\C;\;df_{t}(\delta_{i}(x))=0,\;\forall i=1,...,m\},$$
where $\Theta_{X}=\langle\delta_{1},...,\delta_{m}\rangle$.
\end{defi}

In \cite{ahmed cidinha tomazella} it was proved that if $LC(X)$ is Cohen-Macaulay then the polar curve $C$ is Cohen-Macaulay.


\begin{prop}
Let $(X,0)$ be a holonomic analytic variety. If any $\mathcal{R}_{X}$-finitely determined function germ has a Morsification whose polar curve is Cohen-Macaulay then $LC(X)$ is Cohen-Macaulay. 
%
\end{prop}

\begin{proof}

 Let $(0,p)\in LC(X)$, then there exists an $\mathcal{R}_{X}$-finitely determined function germ $f\in\mathcal{O}_{n},$ such that $df(0)=p$.
Let $F:(\C^{n}\times\C)\to(\C,0)$, $F(x,t)=f_{t}(x)$,
be a Morsification of $f$. 
By hypothesis $\mathcal{O}_{n+1}/df_{t}(\Theta_{X})$ is Cohen-Macaulay of dimension 1, then by the principle of conservation of number $$\mu_{BR}(f,X)=\sum_{i=0}^{k+1}\sum_{x\in\Sigma f_{t}\cap X_{i}}\dim_{\C}\frac{\mathcal{O}_{n,x}}{df_{t}(\Theta_{X,x})}=\sum_{i=0}^{k+1}\sum_{x\in\Sigma f_{t}\cap X_{i}}m_{i}=\sum_{i=0}^{k+1}n_{i}m_{i}$$
because  if $x\in X_{i}$ is a Morse critical point of $f_{t}$, then $\mu_{BR}(f_{t},X)_{x}=m_{i}$, and by \cite[Proposition 5.8]{bruce roberts},  $LC(X)$ is Cohen-Macaulay.

\end{proof}

When $LC(X)$ is Cohen-Macaulay we have 

$$\mu_{BR}(f,X)=\sum_{x\in\C^{n}}\mu_{BR}(f_{t},X)_{x},$$ where $f_{t}$ is any 1-parameter deformation of $f$.

Our purpose now is to prove similar results for $LC(X)^{-}$. We define the \emph{relative polar curve} by $$C^{-}=\{(x,t)\in C;\;x\in X\},$$
where $C$ is the polar curve of $F$ in $(X,0).$


The prove of the next proposition is similar to the one of \cite[Theorem 3.7]{ahmed cidinha tomazella}.
\begin{prop}\label{vai}
Let $(X,0)$ be a holonomic analytic variety. If $LC(X)^-$ is Cohen-Macaulay  then the relative polar curve of any 1-parameter deformation 
of any $\mathcal{R}_{X}$-finitely determined function germ is Cohen-Macaulay.
\end{prop}

For the converse we need the following lemma, which is the analogous of \cite[Proposition 5.12]{bruce roberts} for the relative Bruce-Roberts number.
\begin{lem}\label{multiplicidade}
Let $(X,0)$ be a holonomic analytic variety and $f\in\mathcal{O}_{n}$. We assume that $f$ restricted to $(X,0)$ is a Morse function. If $x\in X$ is a critical point of $f$ then $\mu_{BR}(f,X)_{x}^{-}=m_{\alpha}$, where $m_{\alpha}$ is the multiplicity of the irreducible component $Y_\alpha$ corresponding to the logarithmic stratum $X_\alpha$ which contains $x$. 
\end{lem}
\begin{proof}
Let $Z_{i}=Y_{i}\setminus\bigcup_{j\neq i}Y_{j}$ where $Y_{i}$ are the irreducible components of $LC(X)$. We know from \cite[Proposition 5.12]{bruce roberts} that  $LC(X)$ is Cohen-Macaulay at points in $Z_{i}$, $i=1,...,k+1$. We see that $LC(X)^-$ coincides locally with $LC(X)$ and hence, $LC(X)^-$ is also Cohen-Macaulay at points in $Z_{i}$, $i=1,...,k+1$.

In fact, let $(0,p)\in Z_{i}$ with $i\neq 0$, then $(x,p)\not\in Y_{0}$. Let $V:= T^{*}\C^{n}\setminus Y_{0}$, which is an open neighborhood of $(x,p)$. Obviously, we have the equality of sets $$LC(X)\cap V=LC(X)^{-}\cap V.$$
Moreover, let $I,\;I^{-}$ and $I_{j}$ be the ideals which define $LC(X)$, $LC(X)^{-}$ and $Y_{j}$, $j=0,...,k+1$, respectively. Then, 
$$I=I_{0}\cap I_{1}\cap...\cap I_{k+1}, \; I^{-}=I_{1}\cap...\cap I_{k+1}\text{ and }I_{0}=\langle p_{1},...,p_{n}\rangle.$$
Since $p\neq 0$, $I_{0}$ is the total ring at $(x,p)$, so we have an equality between germs of complex spaces. 

Finally, we have
$$\mu_{BR}(f,X)_{x}^{-}\stackrel{(*)}=\sum_{i=1}^{k+1}m_{i}n_{i}\stackrel{(**)}=m_{\alpha}.$$
The equalities $(*)$ and $(**)$ are consequences of \cite[Propositions 5.11 and 5.2]{bruce roberts}, respectively.

\end{proof}
We are ready now to prove the converse of the Proposition \ref{vai}.
\begin{prop}
Let $(X,0)$ be a holonomic analytic variety. If the relative polar curve of any 1-parameter deformation of any $\mathcal{R}_{X}$-finitely determined function germ is Cohen-Macaulay then $LC(X)^{-}$ is Cohen-Macaulay. 
\end{prop}
\begin{proof}
Let $(0,p)\in LC(X)^{-}$, then there exists an $\mathcal{R}_{X}$-finitely determined function germ $f\in\mathcal{O}_{n},$ such that $df(0)=p$. Let $F:(\C^{n}\times\C,0)\to(\C,0)$ be a Morsification of $f$ and set $f_{t}(x)=F(x,t)$. 

By hypothesis $\mathcal{O}_{n+1}/df_{t}(\Theta_{X}^{-})$ is Cohen-Macaulay of dimension 1. By the principle of the conservation of the multiplicity,
 $$\dim_{\C}\frac{\mathcal{O}_{n}}{df(\Theta_{X}^{-})}=\sum_{i=1}^{k+1}\sum_{x\in\Sigma f\cap X_{i}}\dim_{\C}\frac{\mathcal{O}_{n,x}}{df_{t}(\Theta_{X,x}^{-})}=\sum_{i=1}^{k+1}\sum_{x\in\Sigma f\cap X_{i}}m_{i}=\sum_{i=1}^{k+1}n_{i}m_{i},$$
because  if $x\in X_{i}$ is a Morse critical point of $f_{t}$, then $\mu_{BR}(f_{t},X)^{-}_{x}=m_{i}$ by  Lemma \ref{multiplicidade}. By \cite[Proposition 5.11]{bruce roberts},  $LC(X)^{-}$ is Cohen-Macaulay.
\end{proof}

As consequence of the previous result,

 $$\mu_{BR}^{-}(f,X)=\sum_{x\in\C^{n}}\mu_{BR}^{-}(f_{t},X)_{x},$$ where $f_{t}$ is any 1-parameter deformation of $f$.

\section{An example with non-isolated singularities}

Given natural numbers $0< k\leq n$, we can see $\mathcal O_k$ as a subring of $\mathcal O_n$ and $\Theta_{k}$ as a subset of $\Theta_{n}$. 
We fix $(x_1,\dots,x_n)$ a system of coordinates in $\mathcal O_n$ and we use $(x_1,\dots, x_k)$ as the coordinate system of $\mathcal O_k$ and $(x_{k+1},\dots,x_n)$ as the one in $\mathcal O_{n-k}$.

Let $(X,0)\subset (\C^{k},0)$ an analytic variety. We denote by $(\tilde{X},0)\subset(\C^{n},0)$ the inclusion of $(X,0)$ in $(\C^{n},0).$ Then 
$\Theta_{\tilde{X}}=\mathcal{O}_{n}\Theta_{X}+\langle \tfrac{\partial}{\partial x_{k+1}},...,\tfrac{\partial}{\partial x_{n}}\rangle$ and $LC(\tilde{X})=LC(X)\times\C^{n-t}.$ 

Consequently if $LC(X)$ is Cohen-Macaulay then $LC(\tilde{X})$ is Cohen-Macaulay. 
In particular, if $(X,0)$ is an IHS then $LC(\tilde{X})$ is Cohen-Macaulay.

%
%
%
%
Let $F\in\mathcal{O}_{n}$ a function germ with isolated singularity such that
$F=f+g$ with $f\in\mathcal{O}_{k}$ and $g\in\mathcal{O}_{n-k}$. It is known by Sebastiani and Thom, \cite{sebastiani thom}, that $\mu(F)=\mu(f)\mu(g)$.
We prove a similar result for the Bruce-Roberts number, $$\mu_{BR}(F,\tilde{X})=\mu(g)\mu_{BR}(f,X).$$ 
%
%
%

\begin{prop}
Let $I$ and $J$ be ideals in $\mathcal{O}_{k}$ and $\mathcal{O}_{n-k}$, respectivelly. If we denote by $I'=I\mathcal O_n$ and $J'=J\mathcal O_n$ the respective induced ideals in $\mathcal O_n$, then
$$\dim_{\C}\frac{\mathcal{O}_{n}}{I'+J'}<\infty\;\mbox{if and only if}\;\dim_{\C}\frac{\mathcal{O}_{k}}{I}<\infty\textup{ and }\dim_{\C}\frac{\mathcal{O}_{n-k}}{J}<\infty.$$
Moreover, if these dimensions are finite then
 $$\dim_{\C}\frac{\mathcal{O}_{n}}{I'+J'}=\left(\dim_{\C}\frac{\mathcal{O}_{k}}{I}\displaystyle\right)\left(\dim_{\C}\frac{\mathcal{O}_{n-k}}{J}\right).$$     
\end{prop}

\begin{proof}
The equivalence follows from
$$ V(I')=V(I)\times\C^{n-t},\; V(J')=\C^{t}\times V(J)\textup{ and } V(I'+J')=V(I)\times V(J).$$
For the equality, by hypothesis there exist positive integer numbers $k',\;k_{i}$ and $k_{j}$ such that $$\mathcal{M}_{n}^{k'}\subset I'+J',\;\mathcal{M}_{k}^{k_{i}}\subset I,\;\mathcal{M}_{n-k}^{k_{j}}\subset J,$$ 
where $\mathcal M_\ell$ is the maximal ideal of $\mathcal O_\ell$.
Let $r=\max\{k',\;k_{i}, \;k_{j}\}$, then

$$
\frac{\mathcal{O}_{n}}{I'+J'}\approx\frac{\frac{\mathcal{O}_{n}}{\mathcal{M}_{n}^{r}}}{\frac{I'+J'}{\mathcal{M}_{n}^{r}}}=\frac{\frac{\C[z_1,z_2]}{\mathcal{M}_{n}^{r}}}{\frac{I''+J''}{\mathcal{M}_{n}^{r}}}=\frac{\C[z_1,z_2]}{I''+J''},$$
where $z_1=(x_1,\dots,x_k)$, $z_2=(x_{k+1},\dots,x_n)$ and $I''$ and $J''$ are the ideals in $\C[z_1,z_2]$ generated by the $r-1$-jets of the generators of $I$ and $J$, respectively. Analogously,
$$\frac{\mathcal{O}_{k}}{I}\approx\frac{\C[z_1]}{I'''} \, \mbox{and} \,
\frac{\mathcal{O}_{n-t}}{J} \approx\frac{\C[z_{2}]}{J'''},$$
where $I'''$ and $J'''$ are the ideals in $\C[z_1]$ and $\C[z_2]$ generated by the $r-1$-jets of the generators of $I$ and $J$, respectively.
Finally, the equality follows from  $$\frac{\C[z_1]}{I'''}\otimes_{\C}\frac{\C[z_2]}{J'''}=\frac{\C[z_1,z_2]}{I''+J''},$$ 
where $\otimes_{\C}$ denotes the tensor product, see  \cite[Proposition 2.7.13]{greuel pfister}.
\end{proof}

We observe that the previous result gives a simpler proof to the equality of \cite{sebastiani thom} about the Milnor numbers. Finally we relate the Bruce-Roberts numbers $\mu_{BR}(F,\tilde{X})$ and $\mu_{BR}(f,X)$.
\begin{cor}
Let $(\tilde{X},0)$, and $(X,0)$ as before, and
\begin{align*}
F:(\C^n,0)&\to(\C,0),\\
(z_1,z_2)&\mapsto f(z_1)+g(z_2) 
\end{align*}
then:
\begin{enumerate}
\item[a)]$F$ is $\mathcal{R}_{\tilde{X}}$-finitely determined if, and only if, $f$ is $\mathcal{R}_{X}$-finitely determined and $g$ has isolated singularity.
\item[b)]If $F$ is $\mathcal{R}_{\tilde{X}}$-finitely determined, $\mu_{BR}(F,\tilde{X})=\mu(g)\mu_{BR}(f,X)$.
\end{enumerate}
\end{cor}
\begin{proof}
It is a consequence of the characterization of $\Theta_{\tilde{X}}$ and the previous theorem.
\end{proof}

\printindex
\end{document}